\documentclass[12pt,a4paper]{article}
\usepackage{amsmath,amsthm,amssymb}
\usepackage{mathrsfs}
\usepackage[all]{xy}
\usepackage{latexsym}
\usepackage{lscape}
\usepackage[margin=3cm]{geometry}

\theoremstyle{definition}
\newtheorem{df}{Definition}[section]

\newtheorem{nota}[df]{Notation}
\theoremstyle{plain}
\newtheorem{thm}[df]{Theorem}
\newtheorem{prop}[df]{Proposition}
\newtheorem{lem}[df]{Lemma}
\newtheorem{cor}[df]{Corollary}

\newcommand{\BR}{{\mathbb R}}

\newcommand{\BZ}{{\mathbb Z}}

\newcommand{\Shom}{{\mathscr{H}\hspace{-1mm}om}}
\newcommand{\Hom}{{\rm Hom}}
\newcommand{\eu}{{\rm eu}}
\newcommand{\pt}{{\rm pt}}
\newcommand{\Db}{{{\sf D}^{\rm b}}}
\newcommand{\Dbc}{{{\sf D}^{\rm b}_{\rm cc}}}
\newcommand{\DbR}{{{\sf D}^{\rm b}_{\BR\text{-c}}}}
\newcommand{\Dbf}{{{\sf D}^{\rm b}_f}}
\newcommand{\bR}{{{\rm R}}}
\newcommand{\Le}{{\rm Le}}
\newcommand{\id}{{\rm id}}
\newcommand{\TK}{{\rm TK}}
\newcommand{\tr}{{\rm tr}}
\newcommand{\Supp}{{\rm Supp}}
\newcommand{\SMH}{{\mathscr{M \hspace{-1mm} H}}}
\newcommand{\MH}{{\mathbb{MH}}}
\newcommand{\tf}{{\widetilde{f}}}
\newcommand{\inv}{{\otimes -1}}

\renewcommand{\SS}{{{\rm SS}}}

\def\d{{\delta}}
\def\o{{\omega}}
\def\D{{\Delta}}
\def\G{{\Gamma}}
\def\L{{\Lambda}}

\makeatletter
  
  \@addtoreset{equation}{section}
\makeatother

\title{Microlocal Lefschetz classes of graph trace kernels}
\author{Yuichi \textsc{Ike}}

\begin{document}

\maketitle

\begin{abstract}
In this paper, we define the notion of graph trace kernels 
as a generalization of trace kernels.
We associate a microlocal Lefschetz class with a graph trace kernel and prove 
that this class is functorial with respect to the composition of kernels.
We apply graph trace kernels to the microlocal Lefschetz fixed point
formula for constructible sheaves.
\end{abstract}

\tableofcontents

\section{Introduction}
In \cite{KS14}, Kashiwara and Schapira introduced the notion of trace 
kernels and the method to associate a microlocal Euler class with a trace kernel.

Let $X$ be a $C^\infty$-manifold and $k$ be a field.
We denote by $\o_X$ the dualizing complex on $X$, that is, $\o_X \simeq 
{\rm or}_X[d_X]$ where ${\rm or}_X$ is the orientation sheaf on $X$ and 
$d_X$ is the dimension of $X$.
Denote by $k_{\D_X}$ and $\o_{\D_X}$ the direct image of $k_X$ and 
$\o_X$ respectively under the diagonal embedding $\d \colon X \hookrightarrow X \times X$.
Let $\pi \colon T^*X \to X$ be the cotangent bundle of $X$.

A \textit{trace kernel} on $X$ is a triplet $(K,u,v)$ where $K$ is an object of the derived category of sheaves $\Db(k_{X \times X})$ and $u,v$ are morphisms 
\begin{align}
u \colon k_{\D_X} \longrightarrow K,\;
v \colon K \longrightarrow \o_{\D_X}.
\end{align}
One can naturally define the microlocal Euler class $\mu \eu(K,u,v)$ as an element of $H^0_\L(T^*X;\mu hom(k_{\D_X},\o_{\D_X})) \simeq H^0_\L(T^*X;\pi^{-1}\o_X)$, where $\L=\SS(K) \cap T^*_{\D_X}(X \times X)$.
Kashiwara and Schapira proved the functoriality of the microlocal Euler classes: the 
microlocal Euler class of the composition $K_1 \circ K_2$ of two trace 
kernels is the composition of the microlocal Euler classes of $K_1$ and 
$K_2$ \cite[Theorem 6.3]{KS14}.
\smallskip

On the other hand, microlocal Lefschetz classes of elliptic pairs 
(Guillermou \cite{Gu96}) and Lefschetz cycles of constructible 
sheaves (Matsui-Takeuchi \cite{MT10}) 
were introduced in order to prove the microlocal fixed point formula. 
For elliptic pairs, see Schapira-Schneiders \cite{ScSn94}.
For recent results on this subject, see also \cite{IMT15} and \cite{RTT13}.

Let us recall the notion of Lefschetz cycles defined in \cite{MT10}.
Let $X$ be a real analytic manifold and $\phi \colon X \to X$ be a morphism of manifolds.
We denote by $\DbR(k_X)$ the bounded derived category of 
$\BR$-constructible sheaves on $X$.
Denote by $\o_{\G_\phi}$ the direct image of $\o_X$ under the graph map 
$\d_\phi \colon X \hookrightarrow X \times X , x \mapsto (x,\phi(x))$.
With a pair $(F,\Phi)$ of $F \in \DbR(k_X)$ and $\Phi \in \Hom(\phi^{-1}F,F)$, one 
can associate a cohomology class $ \mu\Le(F,\Phi,\phi) \in 
H^0_\L(T^*X;\mu_{\D_X}(\o_{\G_\phi}))$, where $\L:=\SS(F) \cap 
T^*_{\D_X}(X \times X) \cap T^*_{\G_\phi}(X \times X)$.
This class is called the Lefschetz cycle or the \textit{microlocal Lefschetz 
class} of the pair $(F,\Phi)$.

The microlocal Lefschetz classes of $\BR$-constructible sheaves can be treated in the same way as the microlocal Euler classes of trace kernels.
Define ${\rm D}_X F:=\bR \Shom(F,\o_X)$, the dual of $F$.
The pair $(F,\Phi)$ gives natural morphisms 
\begin{align}\label{mor:constr}
k_{\D_X} \longrightarrow 
{\rm D}_X F \boxtimes F 
\longrightarrow \o_{\G_\phi}.
\end{align}
The composition of the above morphisms defines a cohomology class in \linebreak
$H^0_\L(T^*X;\mu hom(k_{\D_X},\o_{\G_\phi})) \simeq H^0_\L(T^*X;\mu_{\D_X}(\o_{\G_\phi}))$ and this class coincides 
with $\mu\Le(F,\Phi,\phi)$.
\smallskip

In this paper, we extend the notion of trace kernels so that we can treat fixed point formulas.
Then we associate a microlocal Lefschetz class with such a kernel and prove the functoriality of the class.

For a $C^\infty$-manifold $X$ and a morphism of manifolds $\phi \colon X 
\to X$, a \textit{$\phi$-graph trace kernel} is a triplet $(K,u,v)$ where $K \in \Db(k_{X \times X})$ and $u,v$ are morphisms 
\begin{align}
u \colon k_{\D_X} \longrightarrow K,\;
v \colon K \longrightarrow \o_{\G_\phi}.
\end{align}
One defines the microlocal Lefschetz class $\mu\Le(K,u,v,\phi)$ as an 
element of \linebreak $H^0_\L(T^*X;\mu hom(k_{\D_X},\o_{\G_\phi}))$, where 
$\L:=\SS(K) \cap T^*_{\D_X}(X \times X) \cap T^*_{\G_\phi}(X \times X)$.
By \eqref{mor:constr}, a pair $(F,\Phi)$ of $F \in \DbR(X)$ and $\Phi \in \Hom(\phi^{-1}F,F)$ defines a $\phi$-graph trace kernel.

Our main result is the functoriality of microlocal Lefschetz classes: 
the microlocal Lefschetz class of the composition $K_1 \circ K_2$ of two 
graph trace kernels is the composition of microlocal Lefschetz classes 
of $K_1$ and $K_2$ (for a more precise
statement, see  Theorem \ref{thm:gtkcom}).
As an application, we prove the microlocal Lefschetz fixed point 
formula for constructible sheaves.
\smallskip

Finally, let us explain the difference between our construction and that 
of \cite{KS14}.
In the last section of \cite{KS14}, the authors have remarked that trace 
kernels can be adapted to the Lefschetz fixed point formula for 
constructible sheaves.
However, their construction is not suitable to prove the 
functoriality of the cohomology classes.
Therefore, we extend the notion of trace kernels itself and prove the 
functoriality by using the new framework.

\section{Preliminaries}

\subsection{Review on sheaves}
In this paper, all manifolds are assumed to be real manifolds of class $C^\infty$.
Throughout this paper, let $k$ be a field of characteristic zero.
We follow the notation of \cite{KS90}.

Let $X$ be a manifold.
We denote by $\pi_X \colon T^*X \to X$ its cotangent bundle.
If there is no risk of confusion, we simply write $\pi$ instead of $\pi_X$.
For a submanifold $M$ of $X$, we denote by $T^*_M X$ the conormal bundle to $M$.
In particular, $T^*_X X$ denotes the zero-section of $T^*X$.
We also denote by $a \colon T^*X \to T^*X$ the antipodal map defined by $(x;\xi) \mapsto (x;-\xi)$.
A set $\Lambda \subset T^*X$ is said to be conic if it is invariant by the action of $\BR^+$ on $T^*X$.

Let $f \colon X \to Y$ be a morphism of manifolds.
With $f$ one associates the maps 
\[
\xymatrix{
T^*X \ar[rd]_-{\pi_X} & X \times_Y T^*Y \ar[l]_-{f_d} \ar[r]^-{f_\pi} \ar[d]_-\pi 
& T^*Y \ar[d]^-{\pi_Y}\\
& X \ar[r]^-f& Y.
}
\]

We denote by $k_X$ the constant sheaf on $X$ with stalk $k$ and by $\Db(k_X)$ the bounded derived category of sheaves of $k$-vector spaces on $X$.
One can define Grothendieck's six operations $\bR f_*, f^{-1}, \bR f_!, f^!, \overset{{\rm L}}{\otimes}, \bR\Shom$ as functors of derived categories of sheaves.
Since the functor $\cdot\underset{k_X}{\otimes}\cdot$ is exact, we simply write $\otimes$ instead of $\overset{{\rm L}}{\otimes}$.
One denotes by $\o_X$ the dualizing complex on $X$.
That is, if
$\mathbf{a}_X \colon X \to \pt$ 
denotes the natural map,
set $\o_X:=\mathbf{a}_X^! \: k_{\pt}$.
One also denotes by 
$\o_X^{\otimes -1}:=\bR\Shom(\o_X, k_X)$ the dual of $\o_X$.
More generally, for a morphism $f \colon X \to Y$, we denote by $\o_{X/Y}:=f^!k_Y \simeq \o_X \otimes f^{-1}\o_Y^\inv$ the relative dualizing complex.
Note that $\o_X \simeq {\rm or}_X[d_X]$, where ${\rm or}_X$ is the orientation sheaf on $X$ and $d_X$ is the dimension of $X$.
Recall that there is a natural morphism of functors 
\begin{align}\label{mor:upsh}
\o_{X/Y}\otimes f^{-1}(\cdot) \to f^!(\cdot).
\end{align}
We define the duality functor by 
\begin{align}
{\rm D}_X F:=\bR\Shom(F,\o_X).
\end{align}

For $F \in \Db(k_X)$, we denote by $\Supp(F)$ the support of $F$ and by $\SS(F)$ its micro-support, a closed conic involutive subset of $T^*X$.

For a closed submanifold $M$ of $X$, one denotes by $\mu_M \colon 
\Db(k_X) \to {\sf D}^{\rm b}_{\BR^+}(k_{T^*_M X})$ Sato's microlocalization functor along $M$, where ${\sf D}^{\rm b}_{\BR^+}(k_{T^*_M X})$ is the full subcategory of $\Db(k_{T^*_M X})$ consisting of $\BR^+$-conic objects.
We shall use the functor $\mu hom$ defined in \cite{KS90}.
For $F_1,F_2 \in \Db(k_X)$, one defines the bifunctor 
\begin{align*}
\mu hom  \colon   \Db(k_X)^{\rm op} \times \Db(k_X) 
\to {\sf D}^{\rm b}_{\BR^+}(k_{T^* X}) \\
 \mu hom (F_1,F_2):=\mu_\D \bR\Shom(q_2^{-1}F_1,q_1^!F_2),
\end{align*}
where $q_1$ and $q_2$ are the first and second projections from $X \times X$ and $\D$ is the diagonal.
Note that the support of $\mu hom(F_1,F_2)$ satisfies 
\begin{align}
\Supp(\mu hom (F_1,F_2)) \subset {\rm SS}(F_1) \cap {\rm SS}(F_2).
\end{align}
Furthermore, we have the isomorphism 
\begin{align}
\bR \pi_* \mu hom (F_1,F_2) 
\simeq \mu hom (F_1,F_2)|_X \simeq \bR\Shom(F_1,F_2).
\end{align}

\subsection{Compositions of kernels}
We follow the notation of \cite{KS14}.
The results in this subsection are the same as in Section 3 of \cite{KS14}.
For the convenience of the readers, we give proofs of these results here.

\begin{nota} 
\begin{enumerate}
\item For a manifold $X$, we denote by $\d \colon X \hookrightarrow X \times X$ the 
      diagonal embedding and by $\D_X$ the diagonal set of $X \times X$.
\item Let $X_i \: (i=1, 2, 3)$ be manifolds.
For short, we write $X_{ij}:=X_i \times X_j$, $X_{123}:=X_1 \times X_2 \times X_3$, $X_{11223}:=X_1 \times X_1 \times X_2 \times X_2 \times X_3$, etc.
\item Let $\phi_i:X_i \to X_i \; (i=1, 2, 3)$ be morphisms of manifolds.
We write $\phi_{ij}:=\phi_i \times \phi_j:X_{ij} \to X_{ij}$.
\item For simplicity, we shall write $k_i$ instead of $k_{X_i}$ and $\o_i$ instead of $\o_{X_i}$, etc.
We also write $k_{\D_i}$ instead of $k_{\D_{X_i}}$.
\item We denote by $\pi_i$ or $\pi_{ij}$, etc. the projection $T^*X_i \to X_i$ or $T^*X_{ij} \to X_{ij}$, etc.
\item We use the same symbol $q_i$ for the projections $X_{ij} \to X_i$ and $X_{123} \to X_i$. 
We also denote by $q_{ij}$ the projection $X_{123} \to X_{ij}$, 
by $p_i$ the projection $T^*X_{ij} \to T^*X_i$, and by $p_{ij}$ the projection $T^*X_{123} \to T^*X_{ij}$.
\item We denote by $p_{j^a}$ (resp. $p_{ij^a}$) the composition of $p_j$ (resp. $p_{ij}$) and the antipodal map on $T^*X_j$.
\item We denote by $\d_2$ the diagonal embedding $X_{123} \to X_{1223}$.
\end{enumerate}
\end{nota}

Recall the operations of composition of kernels defined in \cite{KS14}.
\begin{df}[\cite{KS14}]
We define the operations of composition of kernels as follows:
\begin{align*}
\underset{2}{\circ} \colon \Db(k_{12}) \times \Db(k_{23}) \to \Db(k_{13}) \\
(K_{12},K_{23}) \mapsto K_{12} \underset{2}{\circ} K_{23} &:= 
\bR {q_{13}}_!\,(q_{12}^{-1}K_{12}\otimes q_{23}^{-1}K_{23}) \\
& \;\simeq \bR {q_{13}}_! \, \d_2^{-1}(K_{12} \boxtimes K_{23}),\\
\underset{2}{*} \colon \Db(k_{12}) \times \Db(k_{23}) \to \Db(k_{13}) \\
(K_{12},K_{23}) \mapsto K_{12} \underset{2}{*} K_{23} &:= 
\bR {q_{13}}_* \, (\o^\inv_{X_{123}/X_{1223}} \otimes \d_2^!(K_{12} 
 \boxtimes K_{23})).
\end{align*}
\end{df}
By \eqref{mor:upsh}, we have a natural morphism $\d_2^{-1}(\cdot) \to \o^\inv_{X_{123}/X_{1223}} \otimes \d_2^!(\cdot)$.
Combining this with the morphism $\bR {q_{13}}_! \to \bR {q_{13}}_*$, we obtain a natural morphism
\begin{align}
K_{12} \underset{2}{\circ} K_{23} \to K_{12} \underset{2}{*} K_{23}.
\end{align}
This is an isomorphism if $p_{12^a}^{-1}(\SS(K_{12})) \cap p_{23}^{-1}(\SS(K_{23}))$ is proper over $T^*X_{13}$.

We define the composition of kernels on cotangent bundles.
\begin{df}[\cite{KS14}]
For kernels on cotangent bundles, we define the composition of kernels as follows:
\begin{align*}
\overset{a}{\underset{2}{\circ}} \colon 
\Db(k_{T^*X_{12}}) \times \Db(k_{T^*X_{23}}) 
& \to \Db(k_{T^*X_{13}}) \\
(K_{12},K_{23}) 
& \mapsto 
K_{12} \overset{a}{\underset{2}{\circ}} K_{23} 
:= 
\bR {p_{13}}_! 
(p_{12^a}^{-1}K_{12} \otimes 
p_{23}^{-1}K_{23}).
\end{align*}
\end{df}

We also define the corresponding operations 
for subsets of cotangent bundles.
Let $A \subset T^*X_{12}$ 
and $B \subset T^*X_{23}$.
We set 
\begin{align*}
A \overset{a}{\underset{2}{\times}} B 
& :=p_{12^a}^{-1}(A) \cap 
p_{23}^{-1}(B),\\
A \overset{a}{\underset{2}{\circ}} B
& :=p_{13}(A \overset{a}{\underset{2}{\times}} B).
\end{align*}

In order to define a composition morphism, we need the following lemma.
Let $X,Y,S$ be manifolds.
Let $q_X \colon X \to S$ and $q_Y \colon Y \to S$ be morphisms.
Assume that 
\begin{align}
X \times_S Y \text{ is a submanifold of }
X \times Y.
\end{align}
Let $j$ be an embedding $X \times_S Y \hookrightarrow X \times Y$.
Noticing that $(X \times_S Y) \times_{(X \times Y)} T^*(X \times Y) 
\simeq T^*X \times_S T^*Y$, we have the following morphisms:
\begin{align}
T^*(X \times_S Y) \overset{j_d}{\longleftarrow} T^*X \times_S T^*Y
\overset{j_\pi}{\longrightarrow} T^*X \times T^*Y.
\end{align}

\begin{lem}\label{lem:compker}{\rm (cf.\ \cite[Proposition 4.4.8]{KS90})}
 For $F_1, G_1 \in \Db(k_X)$ and $F_2,G_2 \in \Db(k_Y)$, 
there is a canonical morphism
{\small \begin{align}
\bR {j_d}_! 
(\mu hom(G_1,F_1)\underset{S}{\boxtimes} \mu hom(G_2,F_2))
\to
\mu hom(j^!(G_1 \boxtimes G_2)
\otimes \o^{\otimes -1}_{X \times_S Y / X \times Y},
F_1 \underset{S}{\boxtimes} F_2).
\end{align}}
\end{lem}
\begin{proof}
First, we construct the morphism when $S=\pt$.
By using the morphism $\mu_M(F) \boxtimes \mu_N(G) \to \mu_{M 
 \times N}(F \boxtimes G)$ \cite[Proposition 4.3.6]{KS90}, 
we obtain a chain of morphisms 
\begin{align*}
& \quad \ \ \mu_{\D_X}\bR\Shom({q_X}_2^{-1}G_1,{q_X}_1^{!}F_1) 
\boxtimes \mu_{\D_Y}\bR\Shom({q_Y}_2^{-1}G_2,{q_Y}_1^{!}F_2) \\
& \to \mu_{\D_X \times \D_Y}(\bR\Shom({q_X}_2^{-1}G_1,{q_X}_1^{!}F_1) 
\boxtimes \bR\Shom({q_Y}_2^{-1}G_2,{q_Y}_1^{!}F_2)) \\
& \to \mu_{\D_X \times \D_Y}(\bR\Shom({q_X}_2^{-1}G_1 \boxtimes {q_Y}_2^{-1}G_2
,{q_X}_1^{!}F_1 \boxtimes {q_Y}_1^{!}F_2) \\
& \, \simeq \mu_{\D_{X \times Y}}\bR\Shom(q_2^{-1}(G_1 \boxtimes G_2),
q_1^!(F_1 \boxtimes F_2)).
\end{align*}

Next, we treat the general case.
Using the morphism  
\begin{align}
\bR {j_d}_! j_\pi^{-1} \mu hom(G,F) 
\to \mu hom(j^! G \otimes \o_{X \times_S Y/X \times Y}^\inv,j^{-1}F) 
\end{align}
from \cite[Proposition 4.4.7]{KS90},
we obtain a chain of morphisms
\begin{align*}
& \quad \ \: \bR {j_d}_! j_\pi^{-1}(\mu hom(G_1,F_1)\boxtimes \mu hom(G_2,F_2)) \\
& \to
 \bR {j_d}_! j_\pi^{-1} \mu hom(G_1 \boxtimes G_2,F_1 \boxtimes F_2) \\
& \to 
\mu hom(j^! (G_1 \boxtimes G_2) \otimes \o_{X \times_S Y/X \times Y}^\inv,
j^{-1}(F_1 \boxtimes F_2)).
\end{align*}
\end{proof}

\begin{prop}[\cite{KS14}]\label{prop:compker}
For $G_1,F_1 \in \Db(k_{12})$ and $G_2,F_2 \in \Db(k_{23})$, there is a composition morphism:
\begin{align}
\mu hom(G_1,F_1) \overset{a}{\underset{2}{\circ}} \mu hom(G_2,F_2)
\to \mu hom(G_1 \underset{2}{*} G_2,F_1 \underset{2}{\circ} F_2).
\end{align}
\end{prop}
\begin{proof}
We shall apply Lemma \ref{lem:compker} for $X_{12} \to X_2$ and $X_{23} \to X_2$.
In this case, $X_{12} \times_{X_2} X_{23} \simeq X_{123}$ and $j$ is the diagonal embedding $X_{123} \hookrightarrow X_{1223}$.
Consider the following commutative diagram: 
\[
\xymatrix{
T^*X_{12} \times T^*X_{23} & T^*X_1 \times T^*X_2 \times T^*X_3 
 \ar[l]_-{(p_{12^a}, p_{23})} 
\ar[d]^{\id \times \widetilde \d \times \id}_{\wr} 
\ar `[d] `[ddd] ^-{p_{13}} [ddd]
\\
T^*X_{12} \times_{X_2} T^*X_{23} \ar[u]^{j_\pi} \ar[d]_{j_d} \ar@{}[rd]|-{\Box}
& T^*X_1 \times T^*_{\D_2}X_{22} \times T^*X_3 \ar[l]_-i \ar[d]^p \\
T^*X_{123} &  T^*X_1 \times X_2 \times T^*X_3 \ar[l]_-{{q_{13}}_d} 
\ar[d]^{{q_{13}}_\pi} \\
& T^*X_{13}
}
\]
where $\widetilde \d$ is defined by $(x_2;\xi_2) \mapsto (x_2,x_2;-\xi_2,\xi_2)$.

By Lemma \ref{lem:compker}, we get a morphism 
\begin{align}\label{mor:muhom}
& \bR {j_d}_! j_\pi^{-1}(\mu hom(G_1,F_1)\boxtimes \mu hom(G_2,F_2)) \notag \\
& \hspace{30mm} \to 
\mu hom(j^!(G_1 \boxtimes G_2)
\otimes \o^{\otimes -1}_{X_{123}/ X_{1223}},j^{-1}(F_1 \boxtimes F_2)).
\end{align}
Set $G:=j^!(G_1 \boxtimes G_2) \otimes \o^{\otimes -1}_{X_{123}/ X_{1223}} 
 \in \Db(k_{123})$ and $F:=j^{-1}(F_1 \boxtimes F_2) \in \Db(k_{123})$.
Combining \eqref{mor:muhom} with the morphism 
\begin{align}
\bR {{q_{13}}_\pi}_! {q_{13}}_d^{-1} \mu hom(G,F)
\to \mu hom(\bR {q_{13}}_* G,\bR {q_{13}}_! F)
\end{align}
from \cite[Proposition 4.4.7]{KS90},
we get a morphism 
\begin{align*}
\bR {{q_{13}}_\pi}_! {q_{13}}_d^{-1}
\bR {j_d}_! j_\pi^{-1}
(\mu hom(G_1,F_1)\boxtimes \mu hom(G_2,F_2))
& \to
 \mu hom(\bR {q_{13}}_* G,\bR {q_{13}}_! F) \\
& \, \simeq 
\mu hom(G_1 \underset{2}{*} G_2,F_1 \underset{2}{\circ} F_2).
\end{align*}
By the above commutative diagram, we have 
\begin{align*}
 \bR {{q_{13}}_\pi}_! {q_{13}}_d^{-1} \bR {j_d}_! j_\pi^{-1}
& \simeq 
\bR {{q_{13}}_\pi}_! \bR p_! i^{-1} j_\pi^{-1} \\
& \simeq 
\bR {{q_{13}}_\pi}_! \bR p_! 
(\id \times \widetilde \d \times \id)_! 
(p_{12^a}, p_{23})^{-1} \\
& \simeq
\bR {p_{13}}_! (p_{12^a}, p_{23})^{-1}.
\end{align*}
Thus, the result follows from the isomorphisms 
\begin{align*}
& \quad \ \bR {{q_{13}}_\pi}_! {q_{13}}_d^{-1}
 \bR {j_d}_! j_\pi^{-1}
(\mu hom(G_1,F_1)\boxtimes \mu hom(G_2,F_2)) \\
& \simeq 
\bR {p_{13}}_! (p_{12^a}, p_{23})^{-1} 
(\mu hom(G_1,F_1)\boxtimes \mu hom(G_2,F_2)) \\
& \simeq 
\bR {p_{13}}_! 
(p_{12^a}^{-1}\mu hom(G_1,F_1)
\otimes p_{23}^{-1} \mu hom(G_2,F_2)).
\end{align*}
\end{proof}

\section{Definition of graph trace kernels}

\subsection{Microlocal homology associated with morphisms}
Let $X$ be a manifold and $\phi \colon X \to X$ be a morphism of manifolds.
We shall identify $X$ with the diagonal $\D_X$ of $X \times X$ and write 
$\D$ instead of $\D_X$ if there is no risk of confusion.
We shall also identify $T^*X$ with $T^*_\D(X \times X)$ by means of the map 
\begin{align}
\d_{T^*X}^a \colon T^*X \to T^*(X \times X), \; (x;\xi) \mapsto (x,x;\xi,-\xi).
\end{align}

We denote by $\d_\phi=(\id_X,\phi) \colon X \hookrightarrow X \times X$ the graph map of $\phi$ and by $\G_\phi=\d_\phi(X)$ the graph of $\phi$.
Set $k_{\G_\phi}:=(\d_\phi)_*k_X$, $\o_{\G_\phi}:=(\d_\phi)_*\o_X$ and $\o^{\otimes -1}_{\G_\phi}:=(\d_\phi)_*\o^{\otimes -1}_X$.
In the case $\phi=\id_X$, we shall write $\d$ for $\d_\phi$ and $k_\D$ for $k_{\G_\phi}$, etc.

\begin{df}
Let $\L$ be a closed conic subset of $T^*X$.
We set
\begin{enumerate}
\item $\SMH_{\hspace{-1mm} \L}(\phi):=\bR \G_\L(\d_{T^*X}^a)^{-1} \mu 
      hom(k_\D,\o_{\G_\phi})$,
\item $\MH_\L(\phi):=\bR  \G(T^*X ; \SMH_{\hspace{-1mm} \L}(\phi))$,
\item $\MH^n_\L(\phi):=H^n(\MH_\L(\phi))$.
\end{enumerate} 
\end{df}

Let $\phi_i \colon X_i \to X_i \; (i=1,2,3)$ be morphisms of manifolds.
We write $\D_i$ for $\D_{X_i} \subset X_{ii}$, etc.

\begin{lem}\label{lem:mh}
We have natural morphisms: 
\begin{enumerate}
\item $\o_{\G_{\phi_{12}}} \underset{22}{\circ} 
      (k_{\G_{\phi_2}}\boxtimes\o_{\G_{\phi_3}}) \to 
      \o_{\G_{\phi_{13}}}$,
\item $k_{\D_{13}} \to k_{\D_{12}}\underset{22}{*} 
      (\o^\inv_{\D_2}\boxtimes k_{\D_3})$.
\end{enumerate}
\end{lem}

\begin{proof}
We denote by $\d_{22}$ the diagonal embedding $X_{112233} \hookrightarrow X_{11222233}$.

\noindent(i) 
We have morphisms
\begin{align*}
 \o_{\G_{\phi_{12}}}
\underset{22}{\circ}(k_{\G_{\phi_2}}\boxtimes\o_{\G_{\phi_3}}) 
& = 
\bR {q_{1133}}_! \d_{22}^{-1}
(\o_{\G_{\phi_1}} \boxtimes \o_{\G_{\phi_2}} 
\boxtimes k_{\G_{\phi_2}}\boxtimes\o_{\G_{\phi_3}}) \\
& \simeq 
\bR {q_{1133}}_! 
(\o_{\G_{\phi_1}} \boxtimes \o_{\G_{\phi_2}} 
\boxtimes \o_{\G_{\phi_3}}) \\
& \to 
\o_{\G_{\phi_{13}}}.
\end{align*}

\noindent(ii) 
The proof is the same as that of \cite[Lemma 4.3]{KS14}.
\end{proof}

\begin{prop}\label{prop:compmu}
 We have a natural composition morphism 
\begin{align}
\mu hom(k_{\D_{12}},\o_{\G_{\phi_{12}}})
\overset{a}{\underset{22}{\circ}}
\mu hom(k_{\D_{23}},\o_{\G_{\phi_{23}}})
\to 
\mu hom(k_{\D_{13}},\o_{\G_{\phi_{13}}}).
\end{align} 
\end{prop}
\begin{proof}
We have the following isomorphisms: 
\begin{align*}
\mu hom(k_{\D_{23}},\o_{\G_{\phi_{23}}})
& \simeq 
\mu hom((\o_2^\inv \boxtimes k_{233}) \otimes k_{\D_{23}},
(\o_2^\inv \boxtimes k_{233}) \otimes \o_{\G_{\phi_{23}}}) \\
& \simeq 
\mu hom(\o_{\D_2}^\inv \boxtimes k_{\D_3},
k_{\G_{\phi_2}} \boxtimes \o_{\G_{\phi_3}}).
\end{align*}
Applying Proposition \ref{prop:compker} and Lemma \ref{lem:mh}, 
we get a chain of morphisms 
\begin{align*}
& \quad \ \, \mu hom(k_{\D_{12}},\o_{\G_{\phi_{12}}})
\overset{a}{\underset{22}{\circ}}
\mu hom(k_{\D_{23}},\o_{\G_{\phi_{23}}}) \\
& \, \simeq
\mu hom(k_{\D_{12}},\o_{\G_{\phi_{12}}})
\overset{a}{\underset{22}{\circ}} 
\mu hom(\o_{\D_2}^\inv \boxtimes k_{\D_3},
k_{\G_{\phi_2}} \boxtimes \o_{\G_{\phi_3}}) \\
& \to 
\mu hom(k_{\D_{12}} \underset{22}{*} 
(\o_{\D_2}^\inv \boxtimes k_{\D_3}), 
\o_{\G_{\phi_{12}}} \underset{22}{\circ} 
(k_{\G_{\phi_2}} \boxtimes \o_{\G_{\phi_3}})) \\
& \to 
\mu hom(k_{\D_{13}},\o_{\G_{\phi_{13}}}).
\end{align*}
\end{proof}

\begin{cor}\label{cor:compmu}
Let $\L_{ij}$ be a closed conic subset of $T^*X_{ij} \:(ij=12,23)$ and assume that 
\begin{align}
\L_{12} 
\overset{a}{\underset{2}{\times}} 
\L_{23}
\to T^*X_{13} 
\text{ is proper}.\label{cond:pro}
\end{align}
Set $\L_{13}:=\L_{12} \overset{a}{\underset{2}{\circ}} \L_{23} \cap (\d^a_{T^*X_{13}})^{-1}(T^*_{\G_{\phi_{13}}}X_{1313})$.
The composition of kernels induces a morphism 
\begin{align}
\overset{a}{\underset{2}{\circ}} \colon 
\MH_{\L_{12}}(\phi_{12}) 
\otimes 
\MH_{\L_{23}}(\phi_{23}) 
\to 
\MH_{\L_{13}}(\phi_{13}).
\end{align} 
In particular, a cohomology class 
$\lambda \in \MH^0_{\L_{12}}(\phi_{12})$ 
defines a morphism 
\begin{align}
\lambda \ \overset{a}{\underset{2}{\circ}} \colon 
\MH_{\L_{23}}(\phi_{23}) 
\to 
\MH_{\L_{13}}(\phi_{13}).
\end{align} 
\end{cor}
\begin{proof}
Noticing that 
\begin{align}
\MH_{\L_{12}}(\phi_{12})
& \simeq 
\bR\G_{\d^a_{T^*X_{12}}\L_{12}}(T^*X_{1122};
\mu hom(k_{\D_{12}},\o_{\G_{\phi_{12}}})), \\
\MH_{\L_{23}}(\phi_{23})
& \simeq 
\bR\G_{\d^a_{T^*X_{23}}\L_{23}}(T^*X_{2233};
\mu hom(k_{\D_{23}},\o_{\G_{\phi_{23}}})), 
\end{align}
we obtain a chain of morphisms
\begin{align*}
& \MH_{\L_{12}}(\phi_{12}) 
\otimes 
 \MH_{\L_{23}}(\phi_{23}) \\
\to 
& \bR\G_{\d^a_{T^*X_{12}}\L_{12}\overset{a}{\underset{2}{\circ}}
\d^a_{T^*X_{23}}\L_{23}}
(T^*X_{1133};\mu hom(k_{\D_{12}},\o_{\G_{\phi_{12}}})
\overset{a}{\underset{22}{\circ}}
\mu hom(k_{\D_{23}},\o_{\G_{\phi_{23}}})) \\
\to 
& \bR\G_{\d^a_{T^*X_{13}}\L_{13}}
(T^*X_{1133};\mu hom(k_{\D_{13}},\o_{\G_{\phi_{13}}})) \\
\simeq \,
& \MH_{\L_{13}}(\phi_{13}).
\end{align*}
Here, the first morphism comes from 
the assumption \eqref{cond:pro} and the second one is given by Proposition \ref{prop:compmu}.
\end{proof}

\subsection{Microlocal Lefschetz classes of graph trace kernels}
Let $\phi \colon X \to X$ be a morphism of manifolds.

\begin{df}
A \textit{$\phi$-graph trace kernel} $(K,u,v)$ is the data of $K \in 
 \Db(k_{X \times X})$ together with morphisms
\begin{align}
k_\D \overset{u}{\longrightarrow} K \quad \text{and} \quad
 K \overset{v}{\longrightarrow} \o_{\G_\phi}.
\end{align} 
\end{df}
In particular, the original trace kernels defined in \cite{KS14} are $\id_X$-graph trace kernels.
If there is no risk of confusion, we simply write $K$ instead of $(K,u,v)$.

For a $\phi$-graph trace kernel $K$, we set 
\begin{align*}
 \SS_{\D,\phi}(K) : & =
\SS(K) \cap T^*_\D(X \times X) \cap T^*_{\G_\phi}(X \times X) \\
& =(\d_{T^*X}^a)^{-1}(\SS(K) \cap T^*_{\G_\phi}(X \times X)).
\end{align*}

\begin{df}
Let $(K,u,v)$ be a $\phi$-graph trace kernel.
\begin{enumerate}
\item The morphism $u$ defines an element $\widetilde{u}$ in $H^0_{\SS(K) \cap T^*_\D(X \times X)}(T^*X;\mu hom(k_\D,K))$.
The \textit{microlocal Lefschetz class} $\mu\Le(K,\phi) \in 
      H^0_{\SS_{\D,\phi}(K)}(T^*X;\mu hom(k_\D,\o_{\G_\phi}))$ of $K$ is 
      the image of $\widetilde{u}$ under the morphism $\mu hom(k_\D,K) \to \mu hom(k_\D,\o_{\G_\phi})$ associated with $v$.
\item Let $\L \subset T^*X$ be a closed conic subset containing $\SS_{\D,\phi}(K)$.
We denote by $\mu\Le_\L(K,\phi)$ the image of $\widetilde{u}$ in $H^0_\L(T^*X ; \mu hom(k_\D,\o_{\G_\phi}))$.
\end{enumerate}
\end{df}

Therefore, we have 
\begin{align}
\mu\Le_\L(K,\phi) \in \MH^0_\L(\phi).
\end{align}
If there is no risk of confusion, we simply write $\mu\Le(K,\phi)$ instead of $\mu\Le_\L(K,\phi)$.

We denote by $\Le(K,\phi)$ the restriction of $\mu\Le(K,\phi)$ to zero-section $X$ of $T^*X$ and call it the \textit{Lefschetz class} of $K$.
Note that 
\begin{align*}
\mu hom(k_\D,\o_{\G_\phi})|_{(X \times X)} 
& \simeq 
\bR\Shom(k_\D,\o_{\G_\phi}) \\
& \simeq (\d_\phi)_*\bR\Shom(k_{(\d_\phi)^{-1}(\D)},\o_X) \\
& \simeq (\d_\phi)_*\bR\G_M(\o_X),
\end{align*}
where $M:=\{x \in X ; \phi(x)=x \}$ is the fixed point set of $\phi$. 
Since $\bR\G_M(\o_X) \simeq \d^{-1}(\d_\phi)_* \bR\G_M(\o_X)$, we have 
\begin{align}
\Le(K,\phi) \in H^0_M(X;\o_X).
\end{align}
\smallskip
\break

\noindent \textbf{Graph trace kernels for constructible sheaves.}
Denote by $\Dbc(K_X)$ the full triangulated subcategory of $\Db(k_X)$ 
consisting of cohomologically constructible sheaves (see \cite[Section 3.4]{KS90}).

\begin{lem}\label{lem:gtkcc}
Let $F \in \Dbc(k_X)$ and $\Phi \colon \phi^{-1}F \to F$ be a morphism in $\Dbc(k_X)$.
There exist natural morphisms in $\Dbc(k_{X \times X})$
\begin{align}
& k_\D \to {\rm D}_X F \boxtimes F, \label{mor:gtk1}\\
& {\rm D}_X F \boxtimes F \to \o_{\G_\phi}.\label{mor:gtk2}
\end{align}
In other words, a pair $(F,\Phi)$ of an object $F \in \Dbc(k_X)$ 
and a morphism \linebreak $\Phi \colon \phi^{-1}F \to  F$ 
defines naturally a $\phi$-graph trace kernel.
\end{lem}
\begin{proof}
\noindent(i) 
The morphism $\id_F$ induces a morphism
\begin{align}
k_X \to 
\bR\Shom(F,F) \simeq 
\d^!({\rm D}_X F \boxtimes F).
\end{align}
Hence, \eqref{mor:gtk1} is obtained by adjunction.

\noindent(ii) 
Noticing that $\d_\phi^{-1}({\rm D}_X F \boxtimes F) \simeq {\rm D}_X F 
 \otimes \phi^{-1}F$, we have a chain of morphisms 
\begin{align}
\d_\phi^{-1}({\rm D}_X F \boxtimes F) 
\overset{\Phi}{\to} 
{\rm D}_X F \otimes F \to \o_X.
\end{align}
Therefore, \eqref{mor:gtk2} is obtained by adjunction.
\end{proof}

We denote by $\TK_\phi(F,\Phi)$ the $\phi$-graph trace kernel associated 
with the pair $(F,\Phi)$ of $F \in \Dbc(k_X)$ and $\Phi \colon \phi^{-1}F \to F$. 
The graph trace kernel defines a microlocal Lefschetz class $\mu\Le(\TK_\phi(F,\Phi),\phi)$.
We also denote this class by $\mu\Le(F,\Phi,\phi)$.
Note that this construction coincides with that of Lefschetz cycles in \cite{MT10}.
\smallskip

\noindent \textbf{Graph trace kernels over one point.}
Let $X=\pt$.
In this case, a (graph) trace kernel $(K,u,v)$ is the data of $K \in 
\Db(k)$ and morphisms 
\begin{align}
k \overset{u}{\longrightarrow} K \overset{v}{\longrightarrow} k.
\end{align}
The (microlocal) Lefschetz class $\Le(K)$ of $K$ is the image of $1 \in k$ under the morphism $vu$.

Let us denote by $\Dbf(k)$ the full triangulated subcategory of $\Db(k)$ consisting of objects with finite-dimensional cohomology. 
Let $V \in \Dbf(k)$ and $f \colon V \to V$ be a $k$-linear map.
Set $K :=V^* \otimes V$ where $V^* :=\bR{\rm Hom}(V,k)$.
Let $u$ be the dual of the trace morphism $V \otimes V^* \to k$ and $v$ 
be the composition of $\id_{V^*} \otimes f \colon V^* \otimes V \to V^* 
\otimes V$ and the trace morphism.
Then
\begin{align}
\Le(V^* \otimes V)={\rm tr}(f):=\sum_{p \in \BZ} (-1)^p {\rm tr}(H^p(f)).
\end{align}

\section{Main results}

\subsection{Compositions of microlocal Lefschetz classes}
Let $X_1, X_2, X_3$ be manifolds and $\phi_i \colon X_i \to X_i \;(i=1,2,3)$ be morphisms.
For $ij=12, 23$, let $K_{ij}$ be a $\phi_{ij}$-graph trace kernel.

\begin{lem}\label{lem:compopr}
There are natural morphisms 
\begin{align}
K_{12} \underset{22}{\circ} (k_{\G_{\phi_2}}\boxtimes \o_{\G_{\phi_3}})
\to \o_{\G_{\phi_{13}}}, \label{mor:compopr1} \\
k_{\D_{13}} \to 
K_{12} \underset{22}{*} (\o_{\D_2}^{\otimes -1}\boxtimes k_{\D_3}).\label{mor:compopr2}
\end{align} 
\end{lem}
\begin{proof}
\noindent{\rm (i)}
By Lemma \ref{lem:mh} (i), we have a morphism 
\begin{align}
\o_{\G_{\phi_{12}}} \underset{22}{\circ} (k_{\G_{\phi_2}}\boxtimes \o_{\G_{\phi_3}})
 \to \o_{\G_{\phi_{13}}}.
\end{align}
Composing it with the morphism $K_{12} \to \o_{\G_{\phi_{12}}}$, we get \eqref{mor:compopr1}.

\noindent{\rm (ii)}
By Lemma \ref{lem:mh} (ii), we have a morphism 
\begin{align}
k_{\D_{13}} \to k_{\D_{12}} \underset{22}{*} 
(\o_{\D_2}^{\otimes -1}\boxtimes k_{\D_3}).
\end{align} 
Composing it with the morphism $k_{\D_{12}} \to K_{12}$, we get the morphism \eqref{mor:compopr2}.
\end{proof}

Let $\Lambda_{1122} \subset T^*X_{1122}$ be a closed conic subset 
containing $\SS(K_{12})$ and $\L_{23}$ be a closed conic subset of $T^*X_{23}$.
Assume that 
\begin{align}
\L_{1122} 
\overset{a}{\underset{22}{\times}} \d^a_{T^*X_{23}} \L_{23} 
\to T^*X_{1133} 
\text{ is proper}.\label{cond:phipro}
\end{align}
Set 
\begin{align*}
 \L_{12} & := \L_{1122} \cap T^*_{\D_{12}}X_{1122} 
\cap T^*_{\G_{\phi_{12}}}X_{1122}, \\
\L_{1133} & := (\L_{1122} \cap T^*_{\G_{\phi_{12}}}X_{1122})
\overset{a}{\underset{22}{\circ}} \d^a_{T^*X_{23}} \L_{23}, \\
\L_{13} & := \L_{1133} \cap T^*_{\D_{13}}X_{1133} 
\cap T^*_{\G_{\phi_{13}}}X_{1133} \\
& \ = \L_{12} \overset{a}{\underset{2}{\circ}} 
\L_{23} \cap (\d^a_{T^*X_{13}})^{-1}(T^*_{\G_{\phi_{13}}}X_{13}).
\end{align*}
We define a map 
\begin{align}
\Phi_{K_{12}}  \colon 
\MH_{\L_{23}}(\phi_{23}) \to \MH_{\L_{13}}(\phi_{13})
\end{align}
by the chain of morphisms
{\small 
\begin{align*}
 & \hspace{14pt} \MH_{\L_{23}}(\phi_{23}) \\
 & \, \simeq 
\bR\G_{\d^a_{T^*X_{23}}\L_{23}}(T^*X_{2233};
\mu hom(k_{\D_{23}},\o_{\G_{\phi_{23}}})) \\
& \, \simeq 
\bR\G_{\d^a_{T^*X_{23}}\L_{23}}(T^*X_{2233};
\mu hom(\o^{\otimes -1}_{\D_2}\boxtimes k_{\D_3},
k_{\G_{\phi_2}}\boxtimes \o_{\G_{\phi_3}})) \\
& \to 
\bR\G_{\L_{1133}}(T^*X_{1133};
\mu hom(K_{12},\o_{\G_{\phi_{12}}})
\overset{a}{\underset{22}{\circ}}
\mu hom(\o^{\otimes -1}_{\D_2}\boxtimes k_{\D_3},
k_{\G_{\phi_2}}\boxtimes \o_{\G_{\phi_3}})) \\
& \to 
\bR\G_{\L_{1133}}(T^*X_{1133};
\mu hom(K_{12} \underset{22}{*}
(\o^{\otimes -1}_{\D_2}\boxtimes k_{\D_3}),
\o_{\G_{\phi_{12}}} \underset{22}{\circ}
(k_{\G_{\phi_2}}\boxtimes \o_{\G_{\phi_3}}))) \\
& \to 
\bR\G_{\d^a_{T^*X_{13}}\L_{13}}(T^*X_{1133};
\mu hom(k_{\D_{13}},\o_{\G_{\phi_{13}}})) 
\simeq \MH_{\L_{13}}(\phi_{13}).
\end{align*}}
The first morphism is given by $v \colon K_{12} \to \o_{\G_{\phi_{12}}}$ as follows: 
{\small 
\begin{align*}
& \quad \; \
\bR\G_{\d^a_{T^*X_{23}}\L_{23}}(T^*X_{2233};
\mu hom(F,G)) \\
& \to 
\bR\G_{\L_{1122} \cap T^*_{\G_{\phi_{12}}}X_{1122}}
(T^*X_{1122};\mu hom(K,\o_{\G_{\phi_{12}}})) \otimes 
\bR\G_{\d^a_{T^*X_{23}}\L_{23}}(T^*X_{2233};
\mu hom(F,G)) \\
& \to 
\bR\G_{\L_{1133}}(T^*X_{1133};
\mu hom(K_{12},\o_{\G_{\phi_{12}}})
\overset{a}{\underset{22}{\circ}}
\mu hom(F,G)).
\end{align*}}
Here, we set $F:=\o^\inv_{\D_2} \boxtimes k_{\D_3}, G:=k_{\G_{\phi_2}}\boxtimes \o_{\G_{\phi_3}}$, and $K:=K_{12}$ for simplicity, and use \eqref{cond:phipro}.
The second morphism comes from Proposition \ref{prop:compker} and the last one is induced by the morphisms in Lemma \ref{lem:mh} and Lemma \ref{lem:compopr}.

\begin{lem}\label{lem:compphi}
In the situation as above, we have 
\begin{align}
\Phi_{K_{12}} = 
\mu \Le(K_{12},\phi_{12}) \overset{a}{\underset{2}{\circ}}
 \colon \MH_{\L_{23}}(\phi_{23}) \to \MH_{\L_{13}}(\phi_{13}),
\end{align}
where the right hand side is the map given by Corollary $\ref{cor:compmu}$.

\end{lem}
\begin{proof}
Consider the following commutative diagram, where we use the same notation as above: 
$F:=\o^\inv_{\D_2} \boxtimes k_{\D_3}, G:=k_{\G_{\phi_2}}\boxtimes \o_{\G_{\phi_3}}$ and $K:=K_{12}$.
\[
\scalebox{0.8}{
\xymatrix@C=-8mm{
\bR\G_{\d^a_{T^*X_{23}}\L_{23}}(T^*X_{2233};\mu hom(F,G)) 
\ar[dd]_a \ar[dr]^d & \\
& \bR\G_{\L_{1133}}(T^*X_{1133};
\mu hom(k_{\D_{12}},\o_{\G_{\phi_{12}}})
\overset{a}{\underset{22}{\circ}} \mu hom(F,G)) 
\ar[dd]^e \\
\bR\G_{\L_{1133}}(T^*X_{1133};
\mu hom(K,\o_{\G_{\phi_{12}}})
\overset{a}{\underset{22}{\circ}} \mu hom(F,G)) 
\ar[dd]_b \ar[ur] & \\
& \bR\G_{\L_{1133}}(T^*X_{1133};
\mu hom(k_{\D_{12}} \underset{22}{*}F, 
\o_{\G_{\phi_{12}}} \underset{22}{\circ} G)) \\
\bR\G_{\L_{1133}}(T^*X_{1133};
\mu hom(K \underset{22}{*}F, 
\o_{\G_{\phi_{12}}} \underset{22}{\circ} G)) 
\ar[ur]_c. &
}}
\]
\noindent 
By using the morphisms in Lemma \ref{lem:mh}, we get a morphism 
\begin{align}
w \colon  
\bR\G_{\L_{1133}}(T^*X_{1133};\mu hom(k_{\D_{12}} \underset{22}{*} F, \o_{\phi_{12}}\underset{22}{\circ} G))\; \notag\\
\to 
\bR\G_{\d^a_{T^*X_{13}}\L_{13}}
(T^*X_{1133};\mu hom(k_{\D_{13}}, 
\o_{\phi_{13}})).
\end{align}
By construction, the morphism $\Phi_{K_{12}}$ is obtained as the composition of $a, b, c$, and $w$.
On the other hand, the morphism $\mu \Le(K_{12},\phi_{12}) \overset{a}{\underset{2}{\circ}}$ is obtained as the composition of $d, e$, and $w$.
Hence, the result follows from the commutativity of the above diagram.
\end{proof}

For $ij=12,23$, let $\Lambda_{iijj} \subset T^*X_{iijj}$ be a closed conic subset containing $\SS(K_{ij})$.
Assume that 
\begin{align}
\Lambda_{1122} \,\overset{a}{\underset{22}{\times}} \, \Lambda_{2233}
\to T^*X_{1133} \text{ is proper.} \label{cond:comppro}
\end{align}
Set $\L_{1133}=\Lambda_{1122} \,\overset{a}{\underset{22}{\circ}} \, \Lambda_{2233}$ and $\L_{13}=\L_{1133} \cap T^*_{\D_{13}}X_{1133} \cap T^*_{\G_{\phi_{13}}}X_{1133}$.

\begin{thm}\label{thm:gtkcom}
Let $K_{ij}$ be a $\phi_{ij}$-graph trace kernel with $\SS(K_{ij}) \subset \L_{iijj}$.
Assume that \eqref{cond:comppro} holds and set $\widetilde K_{23}:=(\o_2^{\otimes -1} \boxtimes k_{233})\otimes K_{23}$.
Then the following hold.
\begin{enumerate}
\item The object $K_{12}\underset{22}{\circ}\widetilde K_{23}$ is a $\phi_{13}$-graph trace kernel.\\
\item We have $\mu \Le(K_{12}\underset{22}{\circ}\widetilde 
      K_{23},\phi_{13})=\mu \Le(K_{12},\phi_{12}) \overset{a}{\underset{2}{\circ}} \mu \Le(K_{23},\phi_{23})$ 
as an element of $\MH^0_{\L_{13}}(\phi_{13})$.
\end{enumerate}
\end{thm}

\begin{proof}
\noindent(i) 
Tensoring the sequence 
\begin{align}
k_{\D_{23}} \to K_{23} \to 
\o_{\G_{\phi_{23}}}
\end{align} 
with $\o_2^\inv \boxtimes k_{233}$, 
we get a sequence 
\begin{align}
\o_{\D_2}^\inv \boxtimes k_{\D_3} \to 
\widetilde K_{23} \to 
k_{\G_{\phi_2}} \boxtimes \o_{\G_{\phi_3}}.
\end{align}
Combining this with Lemma \ref{lem:compopr}, we have the sequences
\begin{align}
k_{\D_{13}} \to 
K_{12} \underset{22}{*} (\o_{\D_2}^\inv \boxtimes k_{\D_3}) \to
K_{12} \underset{22}{*} \widetilde K_{23} \label{mor:compu}
\end{align}
and 
\begin{align}
K_{12} \underset{22}{\circ} \widetilde K_{23} \to 
K_{12} \underset{22}{\circ} (k_{\G_{\phi_2}}\boxtimes \o_{\G_{\phi_3}}) 
\to \o_{\G_{\phi_{13}}} \label{mor:compv}.
\end{align}
By the assumption \eqref{cond:comppro}, 
we have an isomorphism 
\begin{align}
\alpha  \colon  K_{12}\underset{22}{\circ}\widetilde K_{23}
\overset{\sim}{\longrightarrow} 
K_{12}\underset{22}{*}\widetilde K_{23} \label{mor:compisom}.
\end{align}
Using \eqref{mor:compu}--\eqref{mor:compisom}, we get that 
 $K_{12} \underset{22}{\circ} \widetilde K_{23}$ is a $\phi_{13}$-graph trace kernel.

\noindent(ii)
By Proposition \ref{prop:compker}, under the assumption \eqref{cond:comppro}, 
 $\id_{K_{12}}$ and $\id_{\widetilde K_{23}}$ define a morphism 
\begin{align}
\beta  \colon  K_{12}\underset{22}{*}\widetilde K_{23}
\to 
K_{12}\underset{22}{\circ}\widetilde K_{23}.
\end{align}
This morphism is the inverse of the morphism $\alpha$ of \eqref{mor:compisom}.

Now let us consider the following commutative diagram: 
\[
\scalebox{0.9}{
\xymatrix{
k_{\D_{13}} \ar[r] \ar[rd]& 
K_{12} \underset{22}{*} (\o_{\D_2}^{\otimes -1}\boxtimes k_{\D_3}) 
\ar[r] \ar[d] &
K_{12} \underset{22}{\circ} (k_{\G_{\phi_2}}\boxtimes \o_{\G_{\phi_3}}) 
\ar[r] &
\o_{\G_{\phi_{12}}} 
\underset{22}{\circ} (k_{\G_{\phi_2}}\boxtimes \o_{\G_{\phi_3}})
\\
& 
K_{12} \underset{22}{*} \widetilde K_{23} 
\ar@<-1.5mm>[r]_{\beta}^{\sim} &
K_{12} \underset{22}{\circ} \widetilde K_{23}. 
\ar[u] \ar[ru] \ar@<-1.5mm>[l]_{\alpha} & 
}}
\]
By the graph trace kernel structure of $K_{12} \underset{22}{\circ} \widetilde K_{23}$, the composition of the bottom arrows and $\gamma \colon \o_{\G_{\phi_{12}}} \underset{22}{\circ}  (k_{\G_{\phi_2}}\boxtimes \o_{\G_{\phi_3}}) \to \o_{\G_{\phi_{13}}}$ defines $\mu\Le(K_{12} \underset{22}{\circ} \widetilde K_{23},\phi_{13})$.
By the construction of the map $\Phi_{K_{12}}$, the composition of the top arrows and $\gamma$ defines $\Phi_{K_{12}}(\mu\Le(K_{23},\phi_{23}))$.
Hence, the result follows from Lemma \ref{lem:compphi}.
\end{proof}

\subsection{Operations on microlocal Lefschetz classes}
Let $X_1$ and $X_2$ be manifolds and $\phi_1 \colon X_1 \to X_1$ and $\phi_2 \colon X_2 \to X_2$ be morphisms of manifolds.
For $i=1,2$, let $K_i$ be a $\phi_i$-graph trace kernel and let $\L_{ii}$ be a closed conic subset of $T^*X_{ii}$ with $\SS(K_i) \subset \L_{ii}$.

Let $f \colon X_1 \to X_2$ be a morphism of manifolds.
Assume that $\phi_2 f =f \phi_1$, that is, the diagram 
\[
\xymatrix{
X_1 \ar[r]^f \ar[d]_{\phi_1} & X_2 \ar[d]^{\phi_2} \\
X_1 \ar[r]_f & X_2
}
\]
commutes.
Since $\G_f \subset (\phi_{12})^{-1}(\G_f)$, we have a natural morphism 
\begin{align}
\Phi \colon (\phi_{12})^{-1}k_{\G_f} \to k_{\G_f}.
\end{align}
Then the pair $(k_{\G_f},\Phi)$ defines naturally a $\phi_{12}$-graph trace kernel $\TK_{\phi_{12}}(k_{\G_f},\Phi)$ by Lemma \ref{lem:gtkcc}.

Set $\tf:=f \times f \colon X_{11} \to X_{22}$.
We identify $X_{1212}$ with $X_{1122}$.
Then we have 
\begin{align*}
 (\o_1^\inv \boxtimes k_{122})\otimes \TK_{\phi_{12}}(k_{\G_f},\Phi)
& \simeq (\o_1^{\otimes -1} \boxtimes k_{122})\otimes \o_{\G_f}\boxtimes k_{\G_f} \\
& \simeq k_{\G_f}\boxtimes k_{\G_f} \\
& \simeq k_{\G_{\tf}}.
\end{align*}
We also note that 
\begin{align}
\bR\tf_!\,K_1 
\simeq K_1 \underset{11}{\circ} k_{\G_{\tf}}\,,\quad
\tf^{-1}K_2 \simeq 
k_{\G_{\tf}} \underset{22}{\circ} K_2.
\end{align}
\smallskip

\noindent \textbf{External product.}
Let $X_2=\pt$. 
We then write $X_2$ instead of $X_3$.
For $i=1,2$, let $\L_i$ be a closed conic subset of $T^*X_i$.
In this case, we have the composition morphism 
\begin{align}
\MH_{\L_1}(\phi_1) \otimes 
\MH_{\L_2}(\phi_2)
\overset{\circ}{\longrightarrow}
\MH_{\L_1 \times \L_2}(\phi_{12}).
\end{align}
Taking the $0$-th cohomology, we have a morphism 
\begin{align}
\MH^0_{\L_1}(\phi_1) \otimes 
\MH^0_{\L_2}(\phi_2)
\overset{\circ}{\longrightarrow}
 \MH^0_{\L_1 \times \L_2}(\phi_{12}).
\end{align}
In this case, we shall denote by $\lambda_1 \times \lambda_2$ instead of $\lambda_1 \circ \lambda_2$.

Set 
\begin{align}
\L_i:=\L_{ii} \cap T^*_{\D_{i}}X_{ii} \quad (i=1,2).
\end{align}
Then by Theorem \ref{thm:gtkcom}, we obtain the following. 
\begin{prop}\label{prop:external}
The object $K_1 \boxtimes K_2$ is a $\phi_{12}$-graph trace kernel and 
\begin{align}
\mu\Le(K_1 \boxtimes K_2,\phi_{12})
=\mu\Le(K_1,\phi_1) \times \mu\Le(K_2,\phi_2).
\end{align} 
\end{prop}

\noindent \textbf{Direct image.}
Let $X_1=\pt$.
We then write $X_1, X_2$ instead of $X_2, X_3$.
Let $\L_1 \subset T^*X_1$ be a closed conic subset.
Assume that 
\begin{align}
f \text{ is proper on } 
\L_1 \cap T^*_{X_1}X_1.
\end{align} 
We set 
\begin{align*}
f_{\mu, \phi_1 \to \phi_2}(\L_1)
:& =\L_1  \overset{a}{\underset{1}{\circ}} T^*_{\G_f}X_{12} 
\cap (\d^a_{T^*X_2})^{-1}(T^*_{\G_{\phi_2}}X_{22}) \\
& = 
f_d f_\pi^{-1}(\L_1) \cap (\d^a_{T^*X_2})^{-1}(T^*_{\G_{\phi_2}}X_{22})
\subset T^*X_2,
\end{align*}
and 
\begin{align}
f_{\mu, \phi_1 \to \phi_2}
:=\overset{a}{\underset{1}{\circ}} 
\mu\Le(k_{\G_f},\Phi,\phi_{12}) \colon 
\MH^0_{\L_1}(\phi_1) \to 
\MH^0_{f_{\mu, \phi_1 \to \phi_2}(\L_1)}(\phi_2).
\end{align}

\begin{prop}\label{prop:direct}
Assume that $\tf$ is proper on $\L_{11} \cap T^*_{X_{11}}X_{11}$ 
and set 
$
\L_1:=\L_{11} \cap T^*_{\D_{1}}X_{11}.
$
Then the object $\bR \tf_!\,K_1$ is a $\phi_2$-graph trace kernel and 
\begin{align}
\mu\Le(\bR\tf_!\,K_1,\phi_2)
=f_{\mu, \phi_1 \to \phi_2}(
\mu\Le(K_1,\phi_1)).
\end{align} 
\end{prop}

\begin{proof}
The assumption implies that $\L_{11} \overset{a}{\underset{11}{\times}} 
 T^*_{\G_\tf}X_{1122} \to T^*X_{22}$ is proper.
By Theorem \ref{thm:gtkcom}, $K_1 \underset{11}{\circ} (\o_1^\inv \boxtimes k_{122}) \otimes \TK_{\phi_{12}}(k_{\G_f},\Phi) 
\simeq K_1 \underset{11}{\circ} k_{\G_\tf} \simeq \bR \tf_! K_1$ is a $\phi_2$-graph trace kernel and we have 
\begin{align}
\mu\Le(\bR\tf_!K_1,\phi_2)=
\mu\Le(K_1,\phi_1) 
\overset{a}{\underset{1}{\circ}} 
\mu\Le(k_{\G_f},\Phi,\phi_{12}).
\end{align}
\end{proof}

\noindent \textbf{Inverse image.}
Let $X_3=\pt$.
Let $\L_2 \subset T^*X_2$ be a closed conic subset.
Assume that 
\begin{align}
 f \text{ is non-characteristic for } 
\L_2.
\end{align} 
We set 
\begin{align*}
 f^{\mu, \phi_1 \to \phi_2}(\L_2)
:&=T^*_{\G_f}X_{12}  \overset{a}{\underset{2}{\circ}} \L_2 
\cap (\d^a_{T^*X_1})^{-1}(T^*_{\G_{\phi_1}}X_{11}) \\
& = 
f_d f_\pi^{-1}(\L_2) 
\cap (\d^a_{T^*X_1})^{-1}(T^*_{\G_{\phi_1}}X_{11})
\subset T^*X_1, 
\end{align*}
and 
\begin{align}
 f^{\mu, \phi_1 \to \phi_2}
:=\mu\Le(k_{\G_f},\Phi,\phi_{12})
\overset{a}{\underset{2}{\circ}} 
 \colon 
\MH^0_{\L_2}(\phi_2) \to 
\MH^0_{f^{\mu, \phi_1 \to \phi_2}(\L_2)}(\phi_1).
\end{align}

\begin{prop}\label{prop:inverse}
Assume that $\tf$ is non-characteristic for $\L_{22}$ and set
$
\L_2:=\L_{22} \cap T^*_{\D_{2}}X_{22}.
$
Then the object $(\o_{X_1/X_2}\boxtimes k_1) \otimes \tf^{-1} K_2$ is a 
 $\phi_1$-graph trace kernel and 
\begin{align}
 \mu\Le((\o_{X_1/X_2}\boxtimes k_1) \otimes \tf^{-1} K_2,\phi_1)
=f^{\mu, \phi_1 \to \phi_2}
(\mu\Le(K_2,\phi_2)).
\end{align} 
\end{prop}

\begin{proof}
The assumption implies that $T^*_{\G_\tf}X_{1122} 
 \overset{a}{\underset{22}{\times}} \L_{22} \to T^*X_{11}$ is proper.
By Theorem \ref{thm:gtkcom}, $\TK_{\phi_{12}}(k_{\G_f},\Phi) 
 \underset{22}{\circ} (\o_2^\inv \boxtimes k_2) \otimes K_2$ is a $\phi_1$-graph trace kernel.
Here, we have isomorphisms 
\begin{align*}
\TK_{\phi_{12}}(k_{\G_f},\Phi) \underset{22}{\circ}
(\o_2^\inv \boxtimes k_2) \otimes K_2
& \simeq 
(\o_1 \boxtimes k_1) \otimes (
k_{\G_\tf}\underset{22}{\circ}
(\o_2^\inv \boxtimes k_2) \otimes K_2) \\
& \simeq 
(\o_1 \boxtimes k_1) \otimes \tf^{-1}(
(\o_2^\inv \boxtimes k_2) \otimes K_2) \\
& \simeq 
(\o_1 \boxtimes k_1) \otimes 
(f^{-1} \o_2^\inv \boxtimes f^{-1} k_2) 
\otimes \tf^{-1} K_2 \\
& \simeq 
(\o_{X_1/X_2} \boxtimes k_1) 
\otimes \tf^{-1} K_2.
\end{align*}
Applying Theorem \ref{thm:gtkcom} again, we have 
\begin{align}
 \mu\Le((\o_{X_1/X_2}\boxtimes k_1) \otimes \tf^{-1} K_2,\phi_1)
=\mu\Le(k_{\G_f},\Phi,\phi_{12}) \overset{a}{\underset{2}{\circ}}
\mu\Le(K_2,\phi_2).
\end{align}
\end{proof}

\noindent \textbf{Tensor product.}
Let $X_1=X_2=X$ and $\phi_1=\phi_2=\phi$.
For $i=1,2$, let $K_i$ be a $\phi$-graph trace kernel and $\L_{ii} \subset T^*(X 
\times X)$ be a closed conic subset satisfying $\SS(K_i) \subset \L_{ii}$.
Assume that 
\begin{align}
\L_{11} \cap \L_{22}^a \subset T^*_{X \times X}(X \times X), \label{cond:tensor}
\end{align}
and set 
\begin{align}
\L_i:=\L_{ii} \cap T^*_{\D}(X \times X) \quad (i=1,2).
\end{align}
Recall that for a morphism $f \colon X \to Y$, we set $\tf:=f \times f \colon X \times X \to Y \times Y$. 
Since $\widetilde \phi \d=\d \phi$, we have a morphism 
\begin{align}
\d^{\mu,\phi \to \widetilde \phi}  \colon 
\MH^0_{\L_1 \times \L_2}(\widetilde \phi) 
\to \MH^0_{\L_1 + \L_2}(\phi).
\end{align}
Composing it with the morphism of external product 
\begin{align}
 \times \colon \MH^0_{\L_1}(\phi) \otimes \MH^0_{\L_2}(\phi)
\to \MH^0_{\L_1 \times \L_2}(\widetilde \phi),
\end{align}
we get a convolution morphism 
\begin{align}
 \star \colon \MH^0_{\L_1}(\phi) \otimes \MH^0_{\L_2}(\phi)
\to \MH^0_{\L_1 + \L_2}(\phi).
\end{align}

\begin{prop}
Assume that \eqref{cond:tensor} holds.
Then the object 
$(\o_X^{\otimes -1}\boxtimes k_X) \otimes K_1 \otimes K_2$ 
is a $\phi$-graph trace kernel and 
\begin{align}
 \mu\Le((\o_X^{\otimes -1}\boxtimes k_X) \otimes K_1 \otimes K_2,\phi)
=\mu\Le(K_1,\phi) \star \mu\Le(K_2,\phi).
\end{align} 
\end{prop}
\begin{proof}
Since we regard $\widetilde \d \colon X \times X \to X \times X \times X 
 \times X$ as the map $(x_1,x_2) \mapsto (x_1,x_2,x_1,x_2)$, we have $\widetilde \d^{-1}(K_1 \boxtimes K_2) \simeq K_1 \otimes K_2$.
The assumption implies $\d$ is non-characteristic for $\L_1 \times \L_2$.
Thus, the result follows from Proposition \ref{prop:external} and \ref{prop:inverse}, since $\o_{X/X \times X} \simeq \o_X^\inv$.
\end{proof}

\subsection{Application to Lefschetz fixed point formula for
constructible sheaves}
Let $X$ be a real analytic manifold and $\phi_X \colon X \to X$ be a morphism of manifolds.
We denote by $\DbR(k_X)$ the full triangulated subcategory of $\Db(k_X)$ consisting of $\BR$-constructible complexes.
Since $\BR$-constructible complexes are 
cohomologically constructible, a pair 
$(F,\Phi)$ of an object $F \in \DbR(k_X)$ and a morphism $\Phi \colon \phi_X^{-1}F \to F$ gives naturally a $\phi_X$-graph trace kernel $\TK_{\phi_X}(F,\Phi)$.

Let $Y$ be another real analytic manifold 
and $\phi_Y \colon Y \to Y$ be a morphism.
Let $f \colon X \to Y$ be a morphism of manifolds which satisfies $\phi_Y f=f \phi_X$.
Then we have a natural morphism 
\begin{align*}
\phi_Y^{-1}\bR f_*F 
& \to 
\phi_Y^{-1}\bR f_*
\bR{\phi_X}_* \phi_X^{-1}F \\
& \overset{\Phi}{\to} 
\phi_Y^{-1} \bR {\phi_Y}_* \bR f_*F \\
& \to 
\bR f_*F.
\end{align*}
We denote this morphism by $\bR f_* \Phi$.

\begin{prop}\label{prop:lcdirect}
Assume that $f$ is proper on $\Supp(F)$.
Then 
\begin{align}
\mu\Le(\bR f_*F,\bR f_*\Phi,\phi_Y) 
=
f_{\mu,\phi_X \to \phi_Y}(\mu\Le(F,\Phi,\phi_X)).
\end{align} 
\end{prop}
\begin{proof}
By assumption, $\bR f_* F \in \DbR(k_Y)$.
Hence, the pair $(\bR f_* F,\bR f_*\Phi)$ gives a $\phi_Y$-trace kernel $\TK_{\phi_Y}(\bR f_* F,\bR f_*\Phi)$.
Then we have an isomorphism 
\begin{align}
\TK_{\phi_Y}(\bR f_* F,\bR f_*\Phi)
\simeq 
\bR \tf_! \TK_{\phi_X}(F,\Phi).
\end{align}
Hence, by Proposition \ref{prop:direct},
\begin{align*}
\mu\Le(\bR f_* F,\bR f_*\Phi,\phi_Y)
& = 
\mu\Le(\bR \tf_! \TK_{\phi_X}(F,\Phi),\phi_Y) \\
& = 
f_{\mu,\phi_X \to \phi_Y}(\mu\Le(F,\Phi,\phi_X)).
\end{align*}
\end{proof}
Note that the above formula is similar to that of \cite{MT10}.

Applying Proposition \ref{prop:lcdirect} for 
$Y=\pt$ and the natural morphism $f=\mathbf{a} \colon X \to \pt$, 
we obtain

\begin{cor}
Assume that $\Supp(F)$ is compact.
Then 
\begin{align}
\tr(F,\Phi)=
\mathbf{a}_\mu
(\mu\Le(F,\Phi,\phi_X)),
\end{align}
where the left hand side is defined by 
\begin{align}
\tr(F,\Phi):=
\sum_{p \in \BZ} 
(-1)^p \tr \left(H^p(X;F) \to H^p(X;\phi_X^{-1}F) 
\overset{\Phi}{\to} H^p(X;F) \right).
\end{align} 
\end{cor}


\subsection*{Acknowledgments}
The author is grateful to Professor Pierre Schapira for helpful advice.
This work was supported by JSPS KAKENHI Grant Number 15J07993 and 
the Program for Leading Graduate Schools, MEXT, Japan.

\vspace{5mm}

\noindent Graduate School of Mathematical Sciences, the University of Tokyo, Komaba, 
Tokyo, 153-8914, Japan\\
E-mail address: \texttt{ike@ms.u-tokyo.ac.jp}

\end{document}